\documentclass{amsart}
\usepackage{amsfonts,amssymb,amsmath,amsthm}
\usepackage{url}
\usepackage{enumerate}
\usepackage[all]{xy}

\newtheorem{thrm}{Theorem}[section]
\newtheorem{lem}[thrm]{Lemma}
\newtheorem{prop}[thrm]{Proposition}
\newtheorem{cor}[thrm]{Corollary}
\theoremstyle{definition}

\newtheorem{remark}[thrm]{Remark}
\newtheorem{question}[thrm]{Question}
\numberwithin{equation}{section}

\newcommand{\Cs}{$\operatorname{C}^*$-algebra}
\newcommand{\cstar}{\operatorname{C}^*}
\newcommand{\K}{\operatorname{K}}
\newcommand{\KK}{\operatorname{KK}}
\newcommand{\UCT}{\operatorname{UCT}}
\newcommand{\AF}{\operatorname{AF}}
\newcommand{\Cy}{\operatorname{C}}
\newcommand{\supp}{\operatorname{supp}}
\newcommand{\N}{\mathbb N}
\newcommand{\Z}{\mathbb Z}
\newcommand{\R}{\mathbb R}
\newcommand{\C}{\mathbb C}
\newcommand{\cp}{{; cf.}}
\newcommand{\cM}{\mathcal M}
\newcommand{\cN}{\mathcal N}
\newcommand{\cB}{\mathcal B}
\newcommand{\cO}{\mathcal O}
\newcommand{\cP}{\mathcal P}
\newcommand{\elll}{l}

\author{George A. Elliott}
\address{
Department of Mathematics\\
University of Toronto\\
Toronto\\
Canada \ M5S 2E4}
\address{
The Fields Institute for Research in Mathematical Sciences\\
222 College Street\\
Second Floor\\
Toronto\\
Canada \ M5T 3J1}
\email{elliott@math.toronto.edu}
\author{Adam Sierakowski}
\address{
The Fields Institute for Research in Mathematical Sciences\\
222 College Street\\
Second Floor\\
Toronto\\
Canada \ M5T 3J1} 
\email{asierako@fields.utoronto.ca}
\thanks{This research was supported by NSERC}

\keywords{Crossed products, K-theory}
\subjclass{Primary 46L35, Secondary 46L80, 19K99}
\begin{document}

\title[K-theory of certain purely infinite crossed products]{K-theory of certain purely infinite crossed products}

\begin{abstract}
It was shown by R{\o}rdam and the second named author that a countable group $G$ admits an action on a compact space such that the crossed product is a Kirchberg algebra if, and only if, $G$ is exact and non-amenable. This construction allows a certain amount of choice. We show that different choices can lead to different algebras, at least with the free group.
\end{abstract}
\maketitle

\section{Introduction}
The special class of \Cs s, now called \emph{Kirchberg algebras}, that are purely infinite, simple, separable, and nuclear, are of particular interest because of the classification of them (by $\K$- or $\KK$-theory) obtained by Kirchberg and Phillips in the mid 1990’s; see \cite{Kir:fields, Phi:class}. Many of the naturally occurring examples of Kirchberg algebras arise from dynamical systems. The Cuntz algebra $\cO_n$, for example, is stably isomorphic to the crossed product of a stabilized UHF-algebra by an action of the group of integers that scales the trace; see \cite{Cuntz:On}. Crossed products have also been a rich source of examples which are problematic for the classification of \Cs s. One of the most famous examples is R{\o}rdam's simple, separable, nuclear crossed product in the $\UCT$ class which contains both an infinite and a non-zero finite projection, \cite{Ror}. 

Prompted by Choi's embedding of $\cstar_r(\Z_2 * \Z_3)$ into $\cO_2$ (\cite{Choi}), Archbold and Kumjian (independently) proved that there is an action of $\Z_2 * \Z_3$ on the Cantor set such that the corresponding crossed product \Cs{} is isomorphic to $\cO_2$ (see Introduction of \cite{Spi:free}). A number of other constructions of Kirchberg algebras arising as crossed products of abelian \Cs s by hyperbolic groups have appeared in the literature; see \cite{Spi:free, Spi:Fuchsian, LacaSpi:purelyinf, Del}. (Cf.~also \cite{Ell}, in which the group is not hyperbolic.)  

For many cases of interest it is well understood when a crossed product is a Kirchberg algebra. For example, for a countable group $G$ acting on the Cantor set $X$ the crossed product $\Cy(X)\rtimes_r G$ is a Kirchberg algebra precisely when the action is topologically free, amenable, and minimal, and the non-zero projections in $\Cy(X)$ are infinite; see Proposition~\ref{prop21} below.

Recently, it was shown in \cite{Sier:Ror} that a countable group $G$ admits an action on a compact space such that the crossed product is a Kirchberg algebra if, and only if, $G$ is exact and non-amenable. The referee for \cite{Sier:Ror} kindly suggested that it might be interesting to see which Kirchberg algebras can be obtained from a given group $G$. In this paper we give at least a partial answer to this question. 

Let us briefly recall the idea of the construction of \cite{Sier:Ror}, with the emphasis on the parts that make the $\K$-theory computation difficult. Let $G$ be a countable group and let $\beta G$ denote the Stone-Cech compactification of $G$. The strategy is to select a proper sub-\Cs{} $A\rtimes_r G$ of the Roe algebra $\Cy(\beta G) \rtimes_r G$ and then divide out by an ideal to make the algebra simple. The selection involves three main steps:
\begin{enumerate}[(i)]
	\item Recall that an action of $G$ on a totally disconnected compact Hausdorff space $X$ is free if, and only if, for each $e\neq t\in G$ there exists a finite partition $\{p_{i,t}\}_{i\in F}$ of $1$, such that $p_{i,t}\ \bot\ t.p_{i,t}$ (easy exercise). By \cite[Corollary 6.2]{Sier:Ror} such projections $\{p_{i,t}\}$ exist if $X=\beta G$.
	\item Recall that the action of $G$ on a compact Hausdorff space $X$ is said to be amenable if, and only if, for each ${i\in\N}$ there exists a family of positive elements $\{m_{i,t}\}_{t\in G}$ in $\Cy(X)$ such that $\sum_{t\in G}  m_{i,t} =1$ and $\underrightarrow{\lim}_i(\sup_{x\in \beta G}$ $\sum_{t\in G}|m_{i,st}(x)-s.m_{i,t}(x)|)=0$ (see Remark 2.6 following \cite[Definition 2.1]{Del2b}). By \cite[Theorem 4.5]{Del1} and \cite[Theorem 3]{Nar:Oza} such elements $\{m_{i,t}\}$ exist if $X=\beta G$ and $G$ is exact.
	\item Let $p$ be a projection in $\Cy(\beta G)$ not contained in a proper closed two-sided ideal of $\Cy(\beta G)\rtimes_r G$. If $G$ is non-amenable then $p$ is properly infinite; see \cite[Corollary 5.6]{Sier:Ror}. By a density argument there exists a sequence $\{h_{i,p}\}_{i\in \N}$ in $\Cy(\beta G)$ such that $p$ is properly infinite in $\cstar(\{h_{i,p}\})\rtimes_r G$; see  \cite[Lemma 6.6]{Sier:Ror}.
\end{enumerate}
By carefully selecting, according to this procedure, a countable subset of $\Cy(\beta G)$ (and dividing the resulting algebra by a maximal proper $G$-invariant closed two-sided ideal), a free, amenable, minimal action of $G$ on a Cantor set $X$ was obtained in \cite{Sier:Ror} such that every non-zero projection in $\Cy(X)$ is properly infinite. As mentioned in the Introduction (see~\ref{prop21} below) the \Cs{} $\Cy(X)\rtimes_r G$ is a Kirchberg algebra.

In Section~\ref{sec3} we consider the specific case that $G$ is the free group on two generators. We present two specific choices of maps in $\Cy(\beta G)$ such that the algebras they generate divided out by a $G$-invariant ideal give rise to two different crossed products. Perhaps surprisingly the ideal we use is the smallest $G$-invariant closed two-sided ideal in $\elll^\infty(G)\ (\cong \Cy(\beta G))$ containing the finitely supported projections, and the projections we use are only the projections corresponding to cylinder sets and---for the second construction---also the projections corresponding to certain countable unions of cylinder sets. Our main result is
\begin{thrm}\label{mainth0}
Let $G$ denote the free group on two generators. There exist two $G$-invariant sub-$C^*$-algebras of $C( {\beta G \setminus G} )$ such that the corresponding crossed products are different Kirchberg algebras with computable $K$-groups.
\end{thrm}

Since different choices of projections in $\Cy(\beta G)$ lead to different crossed products one might ask if the constructions of \cite{Sier:Ror} can be carried out so as to result in a particular (i.e., identifiable) $\K_0$-group (or particular $\K_1$-group). In  Section~\ref{sec5} we give an affirmative answer to this question, at least for the free groups; see Corollary~\ref{cor1}.

\section{Notation and a preliminary result} \label{sec:notations}
All groups throughout this paper are equipped with the discrete topology. Given a \Cs{} dynamical system $(A,G)$ with $G$ a discrete group, let $A \rtimes_{r} G$ and $A\rtimes G $ denote the \emph{reduced} and the \emph{full crossed product} \Cs s, respectively. Consider the common subalgebra $\Cy_c(G,A)$ of both crossed products consisting of the finite sums $\sum_{t\in G} a_tu_t$, where $a_t \in A$ (only finitely many non-zero), and $t \mapsto u_t$, $t \in G$, is the canonical unitary
representation of $G$ that implements the action of $G$ on $A$. (If $A$ is unital, then each $u_t$ belongs to the crossed product, and in general $u_t$  belongs to the multiplier algebra of the crossed product.) We suppress the canonical inclusion map $A\to A\rtimes_r G $ and view $A$ as being a sub-\Cs{} of $A \rtimes_r G$. 

Recall that the action of $G$ on $\widehat{A}$ (the spectrum of $A$) is said to be \emph{minimal} if $A$ does not contain any non-trivial $G$-invariant closed two-sided ideals, \emph{topologically free} if for any $t_1, \dots, t_n\in G\setminus\{e\}$, the set $\bigcap_{i=1}^n \{x\in \widehat{A} : t_i.x\neq x \}$ is dense in $\widehat{A}$, and \emph{amenable} if there exists a net $(m_i)_{i\in I}$ of continuous maps $x\to m_i^x$ from $\widehat{A}$ to the space Prob$(G)$ such that $\underrightarrow{\lim}_i \|s.m_i^x-m_i^{s.x}\|_1=0$ uniformly on compact subsets of $\widehat{A}$\cp~\cite{Del2b, ArcSpi}. For an action of a discrete group $G$ on an abelian \Cs{} $A$ the crossed product is simple and nuclear if, and only if, the action is minimal, topologically free, and amenable\cp~the proof of Proposition~\ref{prop21}.

Let $a,b$ be positive elements of a \Cs{} $A$. Write $a\precsim b$ if there exists a sequence $(r_n)$ in $A$ such that $r_n^*br_n\to a$. More generally, for $a\in M_n(A)^+$ and $b\in M_m(A)^+$ write $a\precsim b$ if there exists a sequence $(r_n)$ in $M_{m,n}(A)$ with $r_n^*br_n\to a$. For $a\in M_n(A)$ and $b\in M_m(A)$ let $a\oplus b$ denote the element $\textrm{diag}(a,b)\in M_{n+m}(A)$.  

A positive element $a$ in a \Cs{} $A$ is said to be \emph{infinite} if there exists a non-zero positive element $b$ in $A$ such that $a\oplus b\precsim a$. If $a$ is non-zero and if $a\oplus a\precsim a$, then $a$ is said to be \emph{properly infinite}. This extends the usual concepts of infinite and properly infinite projections\cp~\cite[p.\ 642--643]{KirRor}.  

A \Cs{} $A$ is \emph{purely infinite} if there are no characters on $A$ and if for every pair of positive elements $a,b$ in $A$ such that $b$ belongs to the closed two-sided ideal in $A$ generated by $a$, one has that $b\precsim a$. Equivalently, a \Cs{} $A$ is purely infinite if every non-zero positive element $a$ in $A$ is properly infinite\cp~\cite[Theorem 4.16]{KirRor}. 

A \Cs{} $A$ has \emph{real rank zero} if the the set of self-adjoint elements with finite spectrum is dense in the set of all self-adjoint elements (\cite{BroPed:realrank}). Real rank zero is a non-commutative analogue of being totally disconnected (because an abelian \Cs{} $\Cy_0(X)$, where $X$ is a locally compact Hausdorff space, is of real rank zero if and only if $X$ is totally disconnected).

\begin{prop}\label{prop21} Let $G$ be a countable group acting on the Cantor set $X$. The crossed product $C(X)\rtimes_r G$ is a Kirchberg algebra precisely when the action is topologically free, amenable, and minimal, and the non-zero projections in $C(X)$ are infinite.
\end{prop}
\begin{proof}
The crossed product is simple and nuclear if, and only if, the action is topologically free, amenable, and minimal (\cite[Theorem 4.5, Proposition 4.8]{Del1} and \cite[p.\ 124]{ArcSpi}). If the crossed product is purely infinite the non-zero projections in $A$ are infinite. The converse follows from the proof of \cite[Theorem 4.1, $(i)\Rightarrow(iii)$]{Sier:Ror} and \cite[Remark 4.3.7]{Sier:phd} under the assumption of simplicity and nuclearity.
\end{proof}

\section{Kirchberg algebras contained in $\Cy( {\beta G \setminus G} ) \rtimes_r G$}\label{sec3}
\emph{Throughout this section we let $G$ denote the free group on two generators $a,b$}. We will now present two actions of $G$ on the Cantor set such that the corresponding crossed products are different Kirchberg algebras.

Each element in $G$ can be written as a \emph{(reduced) word}, i.e., a sequence finite $z_1\dots z_n$ of letters in $\{a,b,a^{-1},b^{-1}\}$ such that ${z_iz_{i+1}\neq e}$ for $i=1\dots, n-1$. For each $t\in G$ let $\cB(t)$ denote the subset of $G$ consisting of all words starting with $t$ and $|t|$ the length (i.e.\ the number of letters) of the word $t$. 

\subsection{Example 1}
Denote by $A$ the (separable) $G$-invariant sub-\Cs\ of $\elll^\infty(G)$ generated by $\cN=\{1_{\cB(t)}\colon t\in G, t\neq e\}$ and consider the $G$-invariant closed two-sided ideal $I$ in $A$ generated by the projections with finite support. This makes sense since $a.1_{\cB(a^{-1})}=1_{\cB(b)}+1_{\cB(b^{-1})}+1_{\cB(a^{-1})}+1_{\{e\}}$ (in other words, these projections belong to $A$!).

\begin{thrm}\label{mainth1}
The crossed product $A/I\rtimes_r G$ is a Kirchberg algebra in the UCT class with the following K-groups:
$$K_0(A/I\rtimes_r G)\cong \Z^2, \ \ \ K_1(A/I\rtimes_r G)\cong \Z^2.$$
\end{thrm}

Before we give a proof of Theorem~\ref{mainth1} let us establish some preliminary results.

\begin{lem}\label{lem-1}
The $C^*$-algebras $A$ and $A/I$ are unital, separable, abelian AF-algebras of real rank zero. The spaces $\widehat{A}$ and $\widehat{A/I}$ are totally disconnected, compact metric spaces. If the action of $G$ on $\widehat{A/I}$ is minimal and topologically free then the space $\widehat{A/I}$ is a Cantor set.
\end{lem}
\begin{proof}
Any unital abelian \Cs{} $\Cy(X)$ generated by a sequence of projections is an $\AF$-algebra\cp~\cite[Proposition 1.6.10]{Brat}.

Suppose that $G$ acts minimally and topologically freely on $X$. If there were an isolated point in $X$, the orbit of this would have to be all of $X$ by minimality but by compactness it would have to be finite---which would contradict topological freeness.
\end{proof}

\begin{lem}\label{lem1}
Every projection in $A/I$ lifts to a finite sum of projections in $\cN$.
\end{lem}
\begin{proof}
Let $p$ be any non-zero projections in $A/I$. Set $\cN'=\cN\cup \{1_{\{t\}}\colon t\in G\}$. Since $A=\cstar(\cN')$ has real rank zero, $p$ lifts to a projection $q\in A$ such that $p=q+I$; see \cite[Theorem 3.14]{BroPed:realrank}. Since $\cN'$ is a countable set of projections we may denote them by $(p_i)_{i=1}^\infty$. By a density argument and by reindexing the $p_i$'s we can ensure that
$$\|q-\sum _{i=1}^m c_ip_i\|<1/2,$$
for some $m\in \N$ and $c_i\in \C\setminus \{0\}$. Since $q=q^*$ we may assume $c_i\in \R\setminus \{0\}$.

Let us now consider all the non-zero projections we obtain by multiplying $p_1$ or $(1-p_1)$ by $p_2$ or $(1-p_2)$, etc. Except for the projection $(1-p_1)\cdots (1-p_m)$ we denote these projections by $(r_i)_{i=1}^k$. Notice that the $r_i$'s are pairwise orthogonal (taking two different sequences starting with $p_1$ or $(1-p_1)$ followed by $p_2$ or $(1-p_2)$, etc.\ and multiplying them together gives zero). The support of the $r_i$'s equals the support of the $p_i$'s (every $r_i$ corresponds to some product with at least one $p_j$, giving $\supp(r_i)\subseteq \supp(p_j)$; conversely any $x\in \supp(p_j)$ will also be supported by some product of $p_i$'s or $1-p_i$'s, not all of then being $1-p_i$, and hence some element $r_i$, giving $x\in \supp(r_i)$). We conclude that 
$$\sum _{i=1}^m c_ip_i= \sum _{i=1}^k d_ir_i,$$
for some $d_i\in \R$, where each $d_i$ is selected in the canonical way. For clarity we show how to find the $d_i$'s. Fix $j\in \{1,\dots,k\}$. We have $r_j=(\prod_{i\in I_j}p_i)(\prod_{i\in \{1,\dots,m\}\setminus I_j}(1-p_i))$ for some non-empty set $I_j\subseteq \{1,\dots,m\}$. Fix $x\in \supp(r_j)$. We have that 
\begin{align*}
d_j&=\big(\sum _{i=1}^k d_ir_i\big)\big(x\big)=\big(\sum _{i=1}^m c_ip_i\big)\big(x\big)\\
&= \sum _{i\in I_j} c_ip_i(x)+\sum _{i\in \{1,\dots,m\}\setminus I_j} c_ip_i(x)= \sum _{i\in I_j} c_i.
\end{align*}	
We have now shown that
$$\|q-\sum _{i=1}^k d_ir_i\|<1/2, \ \ \ r_ir_j=0, \ \ \ i\neq j, \ \ \ d_i\in \R\setminus \{0\}.$$
Fix $j\in \{1,\dots,k\}$. We have $|q(x)-\sum _{i=1}^k d_ir_i(x)|<1/2$ for any $x\in G$. Using this inequality for $x\in \supp(r_j)$ we conclude that either $|0-d_j|<1/2$ or $|1-d_j|<1/2$. We can therefore define $e_j$ to be the choice of either $0$ or $1$ such that $|e_j-d_j|<1/2$. From
$$\|q-\sum _{i=1}^k e_ir_i\|\leq \|q-\sum _{i=1}^k d_ir_i\|+\|\sum _{i=1}^k (d_i-e_i)r_i\|<1,$$
we conclude that $q=\sum _{i=1}^k e_ir_i$\cp~\cite[p.\  23]{RorLarLau:k-theory}. Upon reindexing and changing $k$ we may assume that $q=\sum _{i=1}^k r_i$ for some $k\in \N$.
	
With $j\in \{1,\dots,k\}$ we have that $r_j=(\prod_{i\in I_j}p_i)(\prod_{i\in \{1,\dots,m\}\setminus I_j}(1-p_i))$ for some non-empty set $I_j\subseteq \{1,\dots,m\}$. Using that $\cN'\cup\{0\}$ is closed under multiplication and that $1-p_i$ is a sum of projections in $\cN'$ (since $1=\sum_{|t|=n} 1_{\cB(t)}+\sum_{|t|<n}1_{\{t\}}$ for each $n\in \N$) we see that $r_j$ is a finite sum of projections in $\cN'$. We conclude that $q$ is finite sum of projections in $\cN'$. By removing the finitely supported projections from this sum we obtain $q$ with the desired properties.
\end{proof}

\begin{lem}\label{lem2}
Every non-zero projection $p$ in $A/I$ is infinite in $A/I\rtimes_r G$. 
\end{lem}
\begin{proof}
Using Lemma~\ref{lem1} write $p$ as 
$$p=\sum_{i=1}^n p_i +I, \ \ \ p_i\in \cN.$$
The sum of two orthogonal infinite projections is again an infinite projection (if $p=v^*v$, $vv^*\lneqq p$ and $q=w^*w$, $ww^*\lneqq q$ for two orthogonal projections $p,q$ then $p+q=u^*u$, $uu^*\lneqq p+q$ for $u:=v+w$). Hence we may assume that $p=1_{\cB(t)}+I$ for some $t\neq e$. Choose $s\in G$ such that $s.p\lneqq p$. Since
$$p=(u_s p)^*(u_s p), \ \ \ (u_s p)(u_s p)^*=s.p \lneqq p, \ \ \textrm{ and } \ \ u_s p \in A/I\rtimes_r G,$$
we have that $p$ is infinite.
\end{proof}

\begin{lem}\label{lem3}
The action of $G$ on $\widehat{A/I}$ is topologically free. 
\end{lem}
\begin{proof}
By Lemma~\ref{lem-1}, $X=\widehat{A/I}$ is a totally disconnected compact Hausdorff space.
	
It is easy to show that the action of $G$ on $X$ is topologically free if for any non-zero projection $p\in \Cy(X)$ and $t\neq e$ there exists a projection $q\leq p$ such that $t.q\neq q$. (Since $X$ is a Hausdorff space it is enough to show that for $t\neq e$ the set $\{x\in X : t.x\neq x \}$ is dense in $X$. Let $U$ be a neighbourhood of $y\in X$. Using that $X$ is totally disconnected find a projection $p$ such that $\supp(p) \subseteq U$. By assumption there exists a clopen set $V\subseteq \supp(p)$ such that $t.V\neq V$. In particular, there is $x\in V$ such that $t.x\neq x$. We conclude that $U\cap \{x\in X : t.x\neq x \}\neq\textit{{\O}}$.)

Fix a non-zero projection $p\in A/I$ and $t\neq e$. By Lemma~\ref{lem1} there exists a projection $1_{\cB(s)}\in \cN$ such that $1_{\cB(s)}+I\leq p$ for some $s\neq e$. Find $r\in G$ such that 
$$|r|>\max(|t|,|s|),\ \ \  r_1\dots r_{|s|}=s,\ \ \ r_{|r|}s_1\neq e.$$
With $q=1_{\cB(rs)}+I$ it follows that $t.q=tr.1_{\cB(s)} + I \neq r.1_{\cB(s)} + I=q\leq p$. We conclude that the action of $G$ on $\widehat{A/I}$ is topologically free.
\end{proof}

\begin{lem}\label{lem4}
The action of $G$ on $\widehat{A/I}$ is minimal. 
\end{lem}
\begin{proof}
Recall that by Lemma~\ref{lem-1} $X=\widehat{A/I}$ is a totally disconnected compact Hausdorff space.

It is easy to show that the action of $G$ on $X$ is minimal if for any non-zero projection $p\in \Cy(X)$ there exist $s,t\in G$ such that $s.p+t.p\geq 1$. (Suppose that the action fails to be minimal. Then one can find $x \in X$ and a non-empty clopen set $U$ in the complement of the orbit $\cO_x:=G.x$. By assumption there exist $s,t\in G$ such that $s.U\cup t.U=X$. Then $\cO_x$ must intersect $s.U$ or $t.U$ or both. By $G$-invariance the orbit $\cO_x$ must intersect $U$; contrary to the choice of $U$.)

Fix a non-zero projection $p\in A/I$. By Lemma~\ref{lem1} there exists a projection $1_{\cB(r)}\in \cN$ such that $1_{\cB(r)}+I\leq p$ for some $r\neq e$. By symmetry we may assume that $r$ ends with $a$. With $s=r_{|r|-1}^{-1}\dots r_1^{-1}$ and $t=ab^{-1}a^{-1}s$ we obtain that
\begin{align*}
s.p+t.p &\geq s.1_{\cB(r)}+t.1_{\cB(rba^{-1})}+I\\
 &= 1_{\cB(a)}+ab^{-1}a^{-1}s.1_{\cB(s^{-1}aba^{-1})}+I=1.
\end{align*}
We conclude that the action of $G$ on $\widehat{A/I}$ is minimal.
\end{proof}

\begin{lem}\label{lem5}
The action of $G$ on $\widehat{A/I}$ is amenable. 
\end{lem}
\begin{proof}
Following \cite[Section 4.3]{BroOza} define the $\|\cdot\|_2$ norm on the convolution algebra $\Cy_c(G, A/I)$ by
$$\|S\|_2=\|\langle S,S\rangle\|^{1/2},$$
where $\langle S,T\rangle =\sum_{s\in G} a_s^*b_s$ for $S=\sum_{s\in G} a_su_s$ and $T=\sum_{t\in G}  b_tu_t$ in $\Cy_c(G,A/I)$. By \cite[Definition 4.3.1, Lemma 4.3.7]{BroOza}, the action is amenable if, and only if, there exists a sequence $T_i\in \Cy_c(G,A/I)$ such that $0\leq T_i(t)$, $\langle T_i, T_i\rangle = 1$ and $\|u_s\cdot T_i-T_i\|_2\to 0$. For each $i\in \N$ define $$T_i=\sum_{|t|<i}(\frac{1}{\sqrt{i}}1_{\cB(t)}+I)u_t.$$
We trivially have that $T_i\in \Cy_c(G,A/I)$ and $0\leq T_i(t)$. Using the equality $\sum_{|t|=n}$ $1_{\cB(t)}+I=1$ for $n\in \N$ we get that $\langle T_i, T_i\rangle=\sum_{t\in G}T_i(t)^2=1$. With $\langle u_s\cdot T_i, u_s\cdot T_i\rangle =1$ and $0\leq T_i\cdot T^*_i(s)= \langle T_i, u_s\cdot T_i\rangle$ we obtain that
\begin{align*}
\|u_s\cdot T_i-T_i\|_2^2	&= \|2-\langle T_i, u_s\cdot T_i\rangle-\langle u_s\cdot T_i, T_i\rangle\|\\
&= 2\|1- T_i\cdot T^*_i(s)\|.
\end{align*}
Using that $(n+1) T_{n+1}\cdot T^*_{n+1}(s)- n\ T_{n}\cdot T^*_{n}(s)=1$ for any $n\in \N$ and $s\in G$ with $|s|\leq n$ (this is easy to verify) we get that (for $|s|\leq i$)
\begin{align*}
\|1- T_i\cdot T^*_i(s)\|	&=  \|1-\frac{1}{i}((i-|s|)1+ |s|\ T_{|s|}\cdot T^*_{|s|}(s))\|\\
&= \frac{|s|}{i}\|1- \ T_{|s|}\cdot T^*_{|s|}(s)\|.
\end{align*}
Since $\|\langle T_i, u_s\cdot T_i\rangle\|\leq \|T_i\|_2\|u_s\cdot T_i\|_2 \leq 1$ (\cite[Section 4.3]{BroOza}), we have $0\leq T_{|s|}\cdot T^*_{|s|}(s)\leq 1$. We conclude that $\|u_s\cdot T_i-T_i\|_2^2\leq \frac{2|s|}{i}$. It follows that the action of $G$ on $\widehat{A/I}$ is amenable.
\end{proof}

\begin{lem}\label{lem6}
The $K_0$-group of $A/I\rtimes_r G$ is $\Z^2$.
\end{lem}
\begin{proof}
Pimsner and Voiculescu were able to calculate a six-term exact sequence for the reduced crossed product for any action of $G$ on any \Cs. Considering the action of $G$ on $A/I$, we have the exact sequence
\begin{align*}
\xymatrix{\K_0(A/I)^n  \ar[r]^-{\sigma}     & \K_0(A/I)  \ar[r]^-{\iota} & \K_0(A/I\rtimes_r G) \ar[d]^-{{\bold 0}}\\
\K_1(A/I\rtimes_r G) \ar[u]  & \K_1(A/I) \ar[l] & \K_1(A/I)^n. \ar[l]}
\end{align*}
Here, the map $\sigma$ denotes $\sum_i (1-t_i)$ with $t_1=a, g_2=b$ and $\iota$ is the map induced by the inclusion of $A/I$ in the crossed product $A/I\rtimes_r G$. Recall that separable $\AF$-algebras (in particular $A/I$) have trivial $\K_1$\cp~\cite[p.\  147]{RorLarLau:k-theory}. Using exactness and the first isomorphism theorem we have that
\begin{align*}
\K_0(A/I\rtimes_r G) & = \textrm{ker} ({\bold 0}) = \textrm{im} ({\iota}) \cong  \K_0(A/I) \ /\ \textrm{ker} ({\iota}) =\K_0(A/I) \ /\ \textrm{im} ({\sigma})\\
 & = \{[p]_0-[q]_0\colon p,q\in M_n(A/I) \textrm{ projections}, n\in \N\} \ /\ \textrm{im} ({\sigma})
\end{align*}
Fix a projection $p=[p_{ij}]\in M_n(A/I)$. The projection $p$ is equivalent to a sum of projections in $A/I$. The reason is that the $\AF$ approximation can be made using a finite-dimensional subalgebra of $A/I$ together with the standard matrix units. (Alternatively, use that $\K_0(A/I) \cong \Cy(\widehat{A/I}, \Z)$ by the isomorphism $\dim([p]_0)(x)=\textrm{Tr}(p(x))$, where Tr is the non-normalized trace map\cp~\cite[p.\  54]{RorLarLau:k-theory}. Hence $p\sim p_{11}\oplus\cdots\oplus p_{nn}$\cp~\cite[p.\  119]{RorLarLau:k-theory}.) 

To simplify notation let us suppress $I$. We have that $\K_0(A/I)$ is generated by the elements $\{[p]_0\colon p\in \cN\}$. Using this characterization we have that $\textrm{im} ({\sigma})$ is generated by $\{[p]_0-[a.p]_0, [p]_0-[b.p]_0\colon p\in \cN\}$. For $p\in \cN$ the translates $a^{-1}.p$ and $b^{-1}.p$ are orthogonal sums of elements in $\cN$. Hence the set $\{[p]_0-[a^{-1}.p]_0, [p]_0-[b^{-1}.p]_0\colon p\in \cN\}$ is also contained in $\textrm{im} ({\sigma})$. We obtain that
$$\textrm{im} ({\sigma})= \langle [p]_0-[t.p]_0, \colon p\in \cN , t\in G \rangle\subseteq \K_0(A/I),$$
where $\langle\cdot \rangle$ is the usual notation for ``the smallest group containing". 
Using this $G$-invariance we get that for any projections $p,q\in M_n(A/I)$, 
$$[p]_0-[q]_0= n[1_{\cB(a)}]_0+m[1_{\cB(b)}]_0+ k[1_{\cB(a^{-1})}]_0+l[1_{\cB(b^{-1})}]_0+ \textrm{im} ({\sigma}),$$
for some $n,m,k,l\in \Z$. Since 
\begin{align*}
\sigma(0,-[1_{\cB(b^{-1})}]_0) & = [1_{\cB(a)}]_0+[1_{\cB(a^{-1})}]_0,\\
\sigma(-[1_{\cB(a^{-1})}]_0,0) & = [1_{\cB(b)}]_0+[1_{\cB(b^{-1})}]_0,
\end{align*}
we can simplify the description of $[p]_0-[q]_0$ to contain only $1_{\cB(a)}$ and $1_{\cB(b)}$. These two projections are linearly independent in $\K_0(A/I)$, which is isomorphic to $\Cy(\widehat{A/I}, \Z)$ by the isomorphism $\dim([p]_0)(x)=\textrm{Tr}(p(x))$. Moreover, $n[1_{\cB(a)}]_0+m[1_{\cB(b)}]_0$ does not belong to $\textrm{im} ({\sigma})$ for any $n,m\in \Z$. We conclude that
$$\K_0(A/I\rtimes_r G)  =  \{n[1_{\cB(a)}]_0+m[1_{\cB(b)}]_0\colon n,m\in \Z\} \ /\ \textrm{im} ({\sigma})\cong \Z^2.$$
\end{proof}
	
\begin{lem}\label{lem7}
Let $p$ and $q$ be finite sums of projections in $A/I$ with coefficients in $\Z$. Then $p+q=a.p+b.q$ in $A/I$ if, and only if, $p=n1$ and $q=k1$ for some $n,k\in \Z$.
\end{lem}
\begin{proof}
Using Lemma~\ref{lem1} write $p$ and $q$ as 
$$p=\sum_{i=1}^n c_ip_i +I, \ \ \ q=\sum_{i=1}^m d_jq_j +I,$$
with $c_i,d_j\in \Z$ and $p_i,q_j\in \cN$ fulfilling that $p_i\neq p_j$ and $q_i\neq q_j$ for $i \neq j$.

It is enough to show that $p$ is a sum (with coefficients in $\Z$) of elements from $\{1_{\cB(a)}+1_{\cB(b)}+1_{\cB(b^{-1})} $, $1_{\cB(a^{-1})}\}+I$ and that $q$ is a sum of elements from $\{1_{\cB(b)}+1_{\cB(a)}+1_{\cB(a^{-1})}$, $1_{\cB(b^{-1})}\} +I$. Having this we obtain the desired conclusion by counting the number of occurrences of $1_{\cB(a)}+I$, $1_{\cB(b)}+I$, $1_{\cB(a^{-1})}+I$, $1_{\cB(b^{-1})}+I$ in the equation $p+q=a.p+b.q$ in $A/I$.

Let us show that $p$ and $q$ are sums as described above. First, write $p$ and $q$ using projections $p_i=1_{\cB(x_i)}$ and $q_j=1_{\cB(y_j)}$, where all the words $x_i, y_j$ have the same length $k$. Use that
\begin{align*}
1_{\cB(a)} &=1_{\cB(aa)}+1_{\cB(ab)}+1_{\cB(ab^{-1})}\\
 &\vdots& \\
1_{\cB(b^{-1})}&=1_{\cB(b^{-1}a)}+1_{\cB(b^{-1}a^{-1})}+1_{\cB(b^{-1}b^{-1})}
\end{align*}
to ensure that all words $x_i, y_j$ have the same length. Now we most take into account how the condition $p+q=a.p+b.q$ groups words together, in order to obtain the decompositions asserted above.

\emph{Suppose that $k=1$.} The projections $a.p_i$, $b.q_j$ can only have the form
\begin{align*} 
a.p_i&\in\big\{1_{\cB(aa)}, 1_{\cB(ab)}, 1_{\cB(ab^{-1})}, 1_{\cB(a^{-1})}+1_{\cB(b)}+1_{\cB(b^{-1})}\big\}+I,\\
b.q_j&\in\big\{1_{\cB(ba)}, 1_{\cB(bb)}, 1_{\cB(ba^{-1})}, 1_{\cB(a)}+1_{\cB(a^{-1})}+1_{\cB(b^{-1})}\big\}+I.
\end{align*}
Since $p+q=a_11_{\cB(a)}+a_21_{\cB(b)}+a_31_{\cB(a^{-1})}+a_41_{\cB(b^{-1})}+I$ for some $a_i\in \Z$, we must be able to group $1_{\cB(aa)}, 1_{\cB(ab)}, 1_{\cB(ab^{-1})}$ together in order to get $p+q=a.p+b.q$ in $A/I$; this implies that $a_1=a_2=a_4$. Similarly the projections $1_{\cB(ba)}, 1_{\cB(bb)}, 1_{\cB(b^{-1}a)}$ must be grouped together, i.e., must occur with the same coefficient. We conclude that $p$ and $q$ are sums of the form predicted above.

\emph{Suppose that $k\geq 2$.} We will now argue that one can reduce the case $k$ to the case $k-1$ by grouping the projections in an appropriate way. 

Let $A,B$ denote the sets $\{x_i\colon i=1,\dots,n\}, \{y_j \colon j=1,\dots,m\}$. The subsets $a.A$, $b.B\subseteq G$ consist of (reduced) words of length $k-1$, $k+1$. Let $z=z_1\dots z_k$ be a word of length $k$.

(1) Assume that $z\in A$ and $z_1\neq a^{-1}$. Then $az_1z_2\dots z_k\in a.A$. One can argue that $a.A$ must contain any word $az_1\dots z_{k-1}y$ of length $k+1$. To see this notice that $az_1z_2\dots z_k\in a.A$ has length $k+1$ and cannot be found among words in $A\cup B$ of length $k$. Therefore we are forced to group $az_1z_2\dots z_k$ together with 2 other words $az_1z_2\dots z_{k-1}y$ of length $k+1$ in $a.A$ (no words of length $k+1$ in $b.B$ start with $a$). Together, these three words of length $k+1$ in $a.A$ correspond to one word $az_1\dots z_{k-1}$ in $A\cup B$ of length $k$. (We will not use that $az_1\dots z_{k-1}\in A\cup B$. The point is that any $az_1z_2\dots z_{k-1}y$ belongs to $a.A$.) Hence $A$ must contain any word $z_1\dots z_{k-1}y$ of length $k$. Replace these words $x_i$ (not starting with $a^{-1}$) by words $x_i'$ of length $k-1$ and replace the corresponding projections $p_i$ by projections $p_i'$.

(2) Assume that $z\in B$ and $z_1\neq b^{-1}$. By symmetry, $B$ must contain any word $z_1\dots z_{k-1}y$ of length $k$. Replace these words $y_j$ (not starting with $b^{-1}$) by words $y_j'$ of length $k-1$ and replace the corresponding projections $q_j$ by projections $q_j'$.

(3) Assume that $z\in A$ and $z_1= a^{-1}$. One can argue that $A$ must contain any word $z_1\dots z_{k-1}y$ of length $k$. To see this notice that the only words in $a.A\cup b.B$ starting with $a^{-1}$ have length $k-1$. Hence $z_1z_2\dots z_k$ cannot be found among words in $a.A\cup b.B$. Therefore we are forced to group $z_1z_2\dots z_k$ together with two other words $z_1z_2\dots z_{k-1}y$ of length $k$ in $A$ (no words of length $k$ in $B$ start with $a^{-1}$; if such a word existed it was eliminated in part (2)). Together, these three words of length $k$ in $A$ correspond to one word $z_1z_2\dots z_{k-1}$ in $a.A\cup b.B$ of length $k-1$. (We will not use that $z_1\dots z_{k-1}\in a.A\cup b.B$. The point is that any $z_1\dots z_{k-1}y$ belongs to $A$.) Hence $A$ must contain any word $z_1\dots z_{k-1}y$ of length $k$. Replace these words $x_i$ (starting with $a^{-1}$) by words $x_i'$ of length $k-1$ and replace the corresponding projections $p_i$ by projections $p_i'$.

(4) Assume that $z\in B$ and $z_1= b^{-1}$. By symmetry, $B$ must contain any word $z_1\dots z_{k-1}y$ of length $k$. Replace these words $y_j$ (starting with $b^{-1}$) by words $y_j'$ of length $k-1$ and replace the corresponding projections $q_j$ by projections $q_j'$. 

(1)--(4) We conclude that we can write $p$ and $q$ as 
$$p=\sum_{i=1}^n c'_ip'_i, \ \ \ q=\sum_{i=1}^m d'_jq'_j,$$
with $c'_i,d'_j\in \Z$ and $p'_i,q'_j\in \cN$ fulfilling that $p'_i\neq p'_j$ and $q'_i\neq q'_j$ for $i \neq j$. Furthermore, $p'_i=1_{\cB(x_i')}, q'_j=1_{\cB(y_j')}$ with $x_i',y_j'$ words of length $k-1$. This shows that one can reduce the case $k$ to the case $k-1$ for $k\geq 2$.
\end{proof}
	
\begin{lem}\label{lem8}
The $K_1$-group of $A/I\rtimes_r G$ is $\Z^2$.
\end{lem}
\begin{proof}
Consider the Pimsner-Voiculescu six-term exact sequence for the action of $G$ on $A/I$:
\begin{align*}
\label{diagram0}
\xymatrix{\K_0(A/I)^n  \ar[r]^-{\sigma}  & \K_0(A/I)  \ar[r] & \K_0(A/I\rtimes_r G) \ar[d]\\
\K_1(A/I\rtimes_r G) \ar[u]_-{{\delta}}  & \K_1(A/I) \ar[l]_-{{\bold 0}} & \K_1(A/I)^n. \ar[l]}
\end{align*}

Here the map $\sigma$ denotes $\sum_i (1-t_i)$ with $t_1=a, t_2=b$ and $\delta$ is the index map. Recall that separable $\AF$-algebras (in particular $A/I$) have trivial $\K_1$. Using exactness we obtain that $\delta$ is injective ($\textrm{im} ({\bold 0})=\textrm{ker} (\delta)$) and hence that
\begin{align*}
	\K_1(A/I\rtimes_r G) & \cong \textrm{im} (\delta)= \textrm{ker} (\sigma)\\
	& = \{(h,h')\in \K_0(A/I)^2\colon \sigma(h,h')=0 \}.
\end{align*}	
Fix $(h,h')\in \K_0(A/I)^2$. Find projections $e,f,e',f'$ in $M_m(A/I)$ such that $h=[e]_0-[f]_0$ and $h'=[e']_0-[f']_0$. Suppose that $\sigma(h,h')=0$, i.e., $([e]_0-[f]_0)-a.([e]_0-[f]_0)+([e']_0-[f']_0)-b.([e']_0-[f']_0)=0$. Now note that $\K_0(A/I) \cong \Cy(\widehat{A/I}, \Z)$ by the isomorphism $\dim([e]_0)(x)=\textrm{Tr}(e(x))$, where Tr is the non-normalized trace map. We obtain that
$$p+q=a.p+b.q,$$
with $p=\textrm{Tr}(e)-\textrm{Tr}(f)$ and $q=\textrm{Tr}(e')-\textrm{Tr}(f')$.
By Lemma~\ref{lem3}, $p=n1$ and $q=k1$ with $n,k\in \Z$. It follows that $(h,h')=(n[1]_0,k[1]_0)$. It follows that
\begin{align*}
	\K_1(A/I\rtimes_r G) & \cong \{(h,h')\in \K_0(A/I)^2\colon \sigma(h,h')=0 \}\\
	& = \{(n[1]_0,k[1]_0)\colon  n,k\in \Z\} \cong  \Z^2.
\end{align*}
\end{proof}

\subsection{Example 2}
Let us now turn our attention to the second construction. Let $y_a$ be any word (including $e$) that does not end with $a^{-1}$. Denote by $\cB(y_aa^{\N}b)$ and $\cB(y_aa^{\N}b^{\N}a)$ the unions
$$\bigcup_{k\in \N}\cB(y_a\underbrace{a\dots a}_kb) \ \ \textrm{ and }\ \ \bigcup_{k,l\in\N}\cB(y_a\underbrace{a\dots a}_k\underbrace{b\dots b}_la).$$
Continuing this way define $\cB(y_aa^\N b^\N a^\N\dots a^\N b)$ and $\cB(y_aa^\N b^\N a^\N\dots b^\N a)$ for any strictly positive length of the alternating powers $a^\N$ and $b^\N$. Set
$$\cN_a=\{1_{\cB(y_aa^\N b^\N a^\N\dots a^\N b)}, 1_{\cB(y_aa^\N b^\N a^\N\dots b^\N a)}\},$$
allowing any strictly positive length of the alternating powers $a^\N$ and $b^\N$ (odd length if the last power is $a^\N$ and even if it is $b^\N$) and any word $y_a$ that does not end with $a^{-1}$. In a similar way define $\cN_{b}$, $\cN_{a^{-1}}$ and $\cN_{b^{-1}}$.

Denote by $B$ the (separable) $G$-invariant \Cs\ in $\elll^\infty(G)$ generated by $\cN$ and $\cM=\bigcup_{|t|=1} \cN_t$ and consider the $G$-invariant closed two-sided ideal $I$ of $B$ generated by the projections with finite support. We have the following result:
\begin{thrm}\label{mainth2}
The crossed product $B/I\rtimes_r G$ is a Kirchberg algebra in the UCT class with the following $K$-groups:
$$K_0(B/I\rtimes_r G)\cong 0, \ \ \ K_1(B/I\rtimes_r G)\cong \Z^4.$$
\end{thrm}

Before we give a proof of Theorem~\ref{mainth2} let us establish some preliminary results.

\begin{lem}\label{lem11}
Every projection in $B/I$ lifts to a finite sum of projections of the form
$$(p-p_1)\cdots (p-p_n), \ \ \ 0\leq p_i\lneqq p,$$
with $n\in \N$, $p\in \cN\cup \cM$ and $p_i\in \cM$.
\end{lem}
\begin{proof}
Let $p$ be any non-zero projection in $B/I$. Set $\cN'=\cN\cup\{1_{\{t\}}\colon t\in G\}$ and $\cM'=\cN'\cup\cM$. Since $B=\cstar(\cM')$ has real rank zero, $p$ lifts to a projection $q\in B$ such that $p=q+I$; see \cite[Theorem 3.14]{BroPed:realrank}. Since $\cM'$ is a countable set of projections we may denote them by $(p_i)_{i=1}^\infty$. As in the proof of Lemma~\ref{lem1} there exist numbers $m,k\in \N$ and non-empty subsets $I_1,\dots, I_k\subseteq \{1,\dots,m\}$ such that
$$q=\sum _{j=1}^k r_j,\ \ \ r_j=(\prod_{i\in I_j}p_i)(\prod_{i\in \{1,\dots,m\}\setminus I_j}(1-p_i)), \ \ \ p_i\in \cM'.$$
Fix $j$ such that $r_j \notin I$ (remove other $r_j$'s). Using that $\cM'\cup\{0\}$ is closed under multiplication (this is easy to verify), we have that $p=\prod_{i\in I_j}p_i\in \cN\cup\cM$. 

Fix $i\in I_j^c$ such that $pp_i\notin I$ (remove other $p_i$'s). If $p_i\in \cN'$ one can rewrite $1-p_i$ as a sum of elements from $\cN'$ and $p-pp_i$ as a sum of elements in $\cM'$ (giving the desired form modulo $I$). Hence we may assume $p_i\in \cM$. If $p\in \cM$ then $pp_i\in \cM$ (since $\cM\cup\{0\}$ is closed under multiplication). Hence we may assume $p\in \cN$. If $pp_i\in \cN$ one can rewrite $p-pp_i$ as a sum of elements from $\cN'$ (giving the desired form modulo $I$). Otherwise, $pp_i\in \cM$ as needed. We conclude that we can write $r_j$ as a sum of projections of the form
$$(p-p_1)\cdots (p-p_n), \ \ \ 0\leq p_i\lneqq p,$$
with $n\in \N$, $p\in \cN\cup \cM$ and $p_i\in \cM$.
\end{proof}
	
\begin{lem}\label{lem12}
Every non-zero projection in $B/I$ is infinite in $B/I\rtimes_r G$. 
\end{lem}
\begin{proof}
The sum of two orthogonal infinite projections is again an infinite projection. By Lemma~\ref{lem11} we only need to show that the projection
$$r:=(p-p_1)\cdots (p-p_n), \ \ \ 0\leq p_i\lneqq p,$$
with $n\in \N$, $p\in \cN\cup \cM$ and $p_i\in \cM$, is infinite in $B/I\rtimes_r G$. We do this in two steps.

\emph {Suppose that $p\in \cN$.} We find $t\in G$ such that $t.p\lneqq r$. This implies that $r$ (suppressing $I$) is infinite by
$$r=(u_t r)^*(u_t r), \ \ \ (u_t r)(u_t r)^*=t.r\leq t.p\lneqq r, \ \ \ u_t r\in B/I\rtimes_r G.$$
By symmetry we may assume that $p=1_{\cB(xa)}$. Changing $t$ we may assume that $p=1_{\cB(a)}$. (If $x\neq e$ use that $x^{-1}.p_i\in\cM$ to find $t'\in G$ such that $t'.(x^{-1}.p)\lneqq (x^{-1}.r)$ and define $t:=xt'x^{-1}$.) The key argument in finding $t$ is to notice that each $p_i\lneqq 1_{\cB(a)},$ $i=1,\dots, n$, has support contained in one of the following sets (identifying $a^k$ with $\underbrace{a\dots a}_k$ and $b^k$ with $\underbrace{b\dots b}_k$ for $k\in \N$):
\begin{align*}
&\cB(a^\N b) \\
&\cB(a^{k_1} b^{-\N} a^{-1}), && k_i\in \N\\
&\cB(a^{k_1} b^{-k_2}a^{\N} b), && k_i\in \N\\
&\cB(a^{k_1} b^{-k_2} a^{k_3} b^{-\N} a^{-1}), && k_i\in \N\\
&\cB(a^{k_1} b^{-k_2} a^{k_3} b^{-k_4} a^{\N} b), && k_i\in \N\\
&\cB(a^{k_1} b^{-k_2} a^{k_3} b^{-k_4} a^{k_5} b^{-\N} a^{-1}), && k_i\in \N\\
&\vdots
\end{align*}
Notice that these sets are pairwise disjoint. Hence for $N\in \N$ large enough one can find a projection
$$q=1_{\cB(a^{k_1} b^{-k_2} a^{k_3}\dots b^{-k_N} a^{\N} b)}, \ \ \ k_i = 1,$$
that is pairwise orthogonal to each $p_i$, $i=1,\dots, n$. Defining $t=a^{k_1} b^{-k_2}\dots$ $b^{-k_N} a b,$ we obtain that
\begin{align*}
t.p & = 1_{\cB(a^{k_1} b^{-k_2} a^{k_3}\dots b^{-k_N} a ba)}\\
 & \lneqq q \ \leq \ p-p_i
\end{align*}
for each $i=1,\dots, n$.  We conclude that $t.p\lneqq r$.

\emph {Suppose that $p\in \cM$.} Let us find $t\in G$ such that $t.p\lneqq r$. This implies that $r$ (suppressing +$I$) is infinite. By symmetry we may assume that $p\in \cN_a$, meaning that $p$ has a support of the form $\cB(y_a a^\N b^\N a^\N\dots a^\N b)$ or $\cB(y_aa^\N b^\N a^\N\dots b^\N a)$, with $y_aa$ being a reduced word. To prevent redundancy we only consider $p=1_{\cB(y_aa^\N b^\N a^\N\dots a^\N b)}$. Changing $t$ we may assume that $y_a=e$. To simplify the notation, write $p$ as $1_{\cB(a^\N xb)}$, where $x=b^\N a^\N\dots b^\N a^\N$ (allowing $x=e$). For later use let $y$ denote the word $b^1 a^1\dots b^1 a^1$ where the powers $\N$ in $x$ are replaced by 1. The key argument in finding $t$ is to notice that each $p_i\lneqq 1_{\cB(a^\N xb)},$ $i=1,\dots, n$, has support contained in one (or two) of the following sets:
\begin{align*}
&\cB(a^\N x b^\N a) \\
&\cB(a^{k} x b), && k\in \N.
\end{align*}
Let $q_0$, $q_k$, $k\in \N$ denote the corresponding projections in $B$. For $N\in \N$ large enough we have that for every $i=1,\dots, n$ there exists a $k_i\in \{0,\dots,N\}$ such that $p_i\leq q_{k_i}$. This implies that
$$(p-q_0)\cdots(p-q_{N})\leq (p-q_{k_i})\leq (p-p_i), \ \ \ i=1,\dots, n.$$
Defining $t:=a^{N+1} y b a^{-1} b $ we obtain that
\begin{align*}
t.p \cdot p & =  1_{\cB(a^{N+1} y b a^{-1} b a^\N xb)\cap \cB(a^\N x b)}=t.p,\\
t.p\cdot q_0 & =  1_{\cB(a^{N+1} y b a^{-1} b a^\N xb)\cap \cB(a^\N x b^\N a)}=0,\\
t.p\cdot q_k & =  1_{\cB(a^{N+1} y b a^{-1} b a^\N xb)\cap \cB(a^{k} x b)}=0, && k\in \{1,\dots,N\},
\end{align*}
Hence $t.p\leq (p-q_0)\cdots(p-q_{N})\leq (p-p_1)\cdots(p-p_n)=r$. (If all $p_i=0$, i.e., $r=p$, one can define $t:=a$.) We conclude that $t.p\lneqq r$.
\end{proof}	
	
\begin{lem}\label{lem13}
The action of $G$ on $\widehat{B/I}$ is topologically free, minimal and amenable.
\end{lem}
\begin{proof}
\emph{1) Topologically free:} Follow the proof of Lemma~\ref{lem3}, but use Lemma~\ref{lem11} and the proof of Lemma~\ref{lem12} instead of Lemma~\ref{lem1} to find the appropriate $1_{\cB(s)}\in \cN$.

\emph{2) Minimal:} Follow the proof of Lemma~\ref{lem4}, but use Lemma~\ref{lem11} and the proof of Lemma~\ref{lem12} instead of Lemma~\ref{lem1} to find the appropriate $1_{\cB(r)}\in \cN$.

\emph{3) Amenable:} Since $A\subseteq B$ the maps $T_i$ from the proof of Lemma~\ref{lem5} ensure that the action of $G$ on $\widehat{B/I}$ is amenable.
\end{proof}

\begin{lem}\label{lem14}
Every projection in $B/I$ has trivial $K_0$-class in ${K_0(B/I\rtimes_r G)}$.
\end{lem}
\begin{proof}
For two orthogonal projections $p,q$ in $B/I$ having trivial $\K_0$-class we have that $[p+q]_0=0$. By Lemma~\ref{lem1} we only need to show that $[r]_0=0$ (suppressing $I$), where
$$r:=(p-p_1)\cdots (p-p_n), \ \ \ 0\leq p_i\lneqq p,$$
with $n\in \N$, $p\in \cN\cup \cM$ and $p_i\in \cM$.
By induction we only need to consider the case $n=1$. To see this notice that for $n=2$ we have
$$[p]_0=[p-r]_0+[r]_0=[p_1]_0+[p_2-p_1p_2]_0+[r]_0.$$
Since $\cM \cup \{0\}$ is closed under multiplication (cf.~proof of Lemma~\ref{lem11}), the case $n=1$ gives that $[p]_0$, $[p_1]_0$, $[p_2-p_1p_2]_0$ have trivial $\K_0$-classes. Hence, also $[r]_0=0$. For $n\geq 3$ we use a similar argument on the equality
\begin{align*}
[p]_0 & = [p-r]_0+[r]_0\\
 & = [p-(p-(p_1-p_1p_2)-(p_2-p_1p_2)-p_1p_2)\prod_{i=3}^n (p-p_i)]_0+[r]_0\\
 & = [((p_1-p_1p_2)+(p_2-p_1p_2)+p_1p_2)\prod_{i=3}^n (p-p_i)]_0+[r]_0\\
 & = [\prod_{i\neq 1}^n (p_1-p_1p_i)]_0+[\prod_{i\neq 2}^n (p_2-p_2p_i)]_0+[\prod_{i=3}^n(p_1p_2-p_1p_2p_i)]_0+[r]_0.\\
\end{align*}
Let us now focus on the case $n=1$. We may assume that $p_1=0$, i.e., $r=p$. If $p_1\neq 0$ simply use $[p]_0=[p_1]_0=0$ on $[p]_0=[p-p_1]_0+[p_1]_0$. 

\emph {Suppose that $p\in \cN$.} By symmetry we may assume that $p=1_{\cB(xa)}$. Since $[x^{-1}.p]_0=[p]_0$ (cf.~proof of Lemma~\ref{lem2}),
we may assume that $p=1_{\cB(a)}$. With $e:= 1_{\cB(b^\N a)}$ it follows that $b^{-1}.e = e + p$. Using the equality $[b^{-1}.e]_0 = [e]_0$ we obtain that $[p]_0=0$.

\emph {Suppose that $p\in \cM$.} By symmetry we may assume that $p\in \cN_a$, meaning that $p$ has a support of the form $\cB(y_a a^\N b^\N a^\N\dots a^\N b)$ or $\cB(y_aa^\N b^\N a^\N\dots b^\N a)$, with $y_aa$ being a reduced word. Since $[y_a^{-1}.p]_0=[p]_0$ we may assume that $y_a=e$. To simplify the notation, write $p$ as $1_{\cB(a^\N x)}$. With $e:=1_{\cB(b^\N a^\N x)}$ we have that $b^{-1}.e = e + p$. Hence $[p]_0=0$.
\end{proof}

\begin{lem}\label{lem15}
The $K_0$-group for $B/I\rtimes_r G$ is $0$.
\end{lem}
\begin{proof}
Following the proof of Lemma~\ref{lem6} we have $\K_0(B/I\rtimes_r G) = \textrm{im} ({\iota})$, where $\iota$ is the map induced by the inclusion of $B/I$ in the crossed product $B/I\rtimes_r G$. Using Lemma~\ref{lem14} together with the fact that $p\sim p_{11}\oplus \dots \oplus p_{nn}$ for a projection $p\in M_n(B/I)$ (cf.~Lemma~\ref{lem6}), we have
\begin{align*}
\K_0(B/I\rtimes_r G) & = \{\iota([p]_0-[q]_0)\colon p,q\in M_n(B/I) \textrm{ projections}, n\in \N\}\\
 & = \{[p]_0-[q]_0\colon p,q\in M_n(B/I) \textrm{ projections}, n\in \N\}\\
 & = 0.
\end{align*}
\end{proof}

\begin{lem}\label{lem16}
Let $p$ and $q$ be finite sums of projections in $B/I$ with coefficients in $\Z$. Denote by $r$ and $r'$ the projections
$$1_{\cB(b)\cup\cB(a^{\N}b)\cup\cB(a^{-1})\setminus\cB(a^{-\N} b^{-1})} \ \ \textrm{ and }\ \ 1_{\cB(a)\cup\cB(b^{\N}a)\cup\cB(b^{-1})\setminus\cB(b^{-\N} a^{-1})}.$$
Then $p+q=a.p+b.q$ in $B/I$ if, and only if, $p=n1+ mr$ and $q=k1+ lr'$ for some $n,m,k,l\in \Z$.
\end{lem}
\begin{proof}
Using Lemma~\ref{lem11} write $p$ and $q$ as 
$$p=\sum_{i=1}^n c_ip_i +I, \ \ \ q=\sum_{i=1}^m d_jq_j +I,$$
with $c_i,d_j\in \Z$ and $p_i,q_j\in \cN\cup \cM$ such that $p_i\neq p_j$ and $q_i\neq q_j$ for $i \neq j$. We use here that $\cN\cup \cM\cup \{0\}$ is closed under multiplication (when multiplying out the product in the statement of Lemma~\ref{lem11}).

It is enough to show that $p$ is a sum (with coefficients in $\Z$) of elements from $\{1_{\cB(a)}+1_{\cB(b)}+1_{\cB(b^{-1})}$, $1_{\cB(a^{-1})}, 1_{\cB(a^\N b)}+1_{\cB(b)}$, $1_{\cB(a^{-\N} b^{-1})}$, $1_{\cB(a^\N b^\N a)}+1_{\cB(b^\N a)}$, $1_{\cB(a^{-\N} b^{-\N} a^{-1})}$, $\dots\}+I$ and that $q$ is a sum of elements from $\{1_{\cB(b)}+1_{\cB(a)}+1_{\cB(a^{-1})}$, $1_{\cB(b^{-1})}$, $1_{\cB(b^\N a)}+1_{\cB(a)}$, $1_{\cB(b^{-\N} a^{-1})}$, $1_{\cB(b^\N a^\N b)}+1_{\cB(a^\N b)}$, $1_{\cB(b^{-\N} a^{-\N} b^{-1})}$, $\dots\}+I$. We do this in a similar way to the proof of Lemma~\ref{lem7}. First we write $p$ and $q$ using projections $p_i=1_{\cB(x_i)}$ and $q_j=1_{\cB(y_j)}$, where all the words $x_i, y_j$ have the same length $k$. To do this we formally allow $x_i, y_j$ to have letters $\{a^{\N},b^{\N},a^{-\N},b^{-\N}\}$ when calculating the length of $x_i$ and $y_j$ and use that 
\begin{align*}
1_{\cB(c^\N x)}&=1_{\cB(\underbrace{c\dots c}_n c^\N x)}+1_{\cB(\underbrace{c\dots c}_n x)}+\cdots+1_{\cB(c x)},\\
1_{\cB(c^{-\N} x)}&=1_{\cB(\underbrace{c^{-1}\dots c^{-1}}_n c^{-\N} x)}+1_{\cB(\underbrace{c^{-1}\dots c^{-1}}_n x)}+\cdots+1_{\cB(c^{-1} x)},
\end{align*}
for $n\in \N$ and $c\in \{a,b\}$. Secondly we use the equality $p+q=a.p+b.q$ to group some words together and exclude other words until we obtain only the elements listed above. Finally we conclude that $p$ and $q$ have the desired form by counting the number of occurrences of $1_{\cB(a)}+I$, $1_{\cB(b)}+I$, $1_{\cB(a^{-1})}+I$, $1_{\cB(b^{-1})}+I$, $1_{\cB(a^\N b)}+I$, $1_{\cB(b^\N a)}+I$, $1_{\cB(a^{-\N} b^{-1})}+I$, $1_{\cB(b^{-\N} a^{-1})}+I$, $\dots$ in the equation $p+q=a.p+b.q$. Let us consider this computation in more detail: We have that 
\begin{align*}
p & =	d_1(1_{\cB(a)}+1_{\cB(b)}+1_{\cB(b^{-1})})+d_21_{\cB(a^{-1})}+ \\
 & \quad a_1(1_{\cB(a^\N b)}+1_{\cB(b)}) + a_2(1_{\cB(a^\N b^\N a)}+1_{\cB(b^\N a)})+ \dots +\\
 & \quad b_11_{\cB(a^{-\N} b^{-1})} + b_21_{\cB(a^{-\N} b^{-\N} a^{-1})}+\dots + I,\\
a.p & =	d_11_{\cB(a)}+d_2(1_{\cB(a^{-1})}+1_{\cB(b)}+1_{\cB(b^{-1})})+ \\
 & \quad a_11_{\cB(a^\N b)}+ a_21_{\cB(a^\N b^\N a)}+ \dots +\\
 & \quad b_1(1_{\cB(a^{-\N} b^{-1})}+1_{\cB(b^{-1})}) + b_2(1_{\cB(a^{-\N} b^{-\N} a^{-1})}+1_{\cB(b^{-\N} a^{-1})})+\dots+I,\\
q & =	d'_1(1_{\cB(b)}+1_{\cB(a)}+1_{\cB(a^{-1})})+d'_21_{\cB(b^{-1})}+ \\
 & \quad a'_1(1_{\cB(b^\N a)}+1_{\cB(a)}) + a'_2(1_{\cB(b^\N a^\N b)}+1_{\cB(a^\N b)})+ \dots +\\
 & \quad b'_11_{\cB(b^{-\N} a^{-1})} + b'_21_{\cB(b^{-\N} a^{-\N} b^{-1})}+\dots+I,\\
b.q & =	d'_11_{\cB(b)}+d'_2(1_{\cB(b^{-1})}+1_{\cB(a)}+1_{\cB(a^{-1})})+ \\
 & \quad a'_11_{\cB(b^\N a)}+ a'_21_{\cB(b^\N a^\N b)}+ \dots +\\
 & \quad b'_1(1_{\cB(b^{-\N} a^{-1})}+1_{\cB(a^{-1})}) + b'_2(1_{\cB(b^{-\N} a^{-\N} b^{-1})}+1_{\cB(a^{-\N} b^{-1})})+\dots+I,
\end{align*}
and we obtain the following coefficients for the individual projections: 
$$
\begin{array}{ l | l | l }                   
   & p+q & a.p+b.q \\
  \hline  
  1_{\cB(a)}+I & d_1+d_1'+a'_1 & d_1+d'_2 \\
  1_{\cB(b)}+I & d_1+a_1+d'_1 & d_2+d'_1 \\
  1_{\cB(a^{-1})}+I & d_2+d'_1 & d_2 + d'_2 + b'_1 \\
  1_{\cB(b^{-1})}+I & d_1+d'_2 & d_2 + b_1 + d'_2 \\
  1_{\cB(a^\N b)}+I & a_1 + a'_2 & a_1 \\
  1_{\cB(b^\N a)}+I &  a_2 + a'_1 & a'_1 \\
  1_{\cB(a^{-\N} b^{-1})}+I & b_1 & b_1 + b'_2 \\
  1_{\cB(b^{-\N} a^{-1})}+I & b'_1 & b'_1 + b_2 \\
  \vdots & \vdots & \vdots
\end{array}
$$

We conclude that $a_n$, $a'_n$, $b_n$ and $b'_n$ are all zero for $n\geq 2$. With $n=d_1$, $m=a_1$, $k=d'_1$, $l=a'_1$ we get that $d_2=n+m$, $b_1=-m$, $d'_2=k+l$ and $b'_1=-l$, giving the desired form of $p$ and $q$.
\end{proof}

\begin{lem}\label{lem17}
The $K_1$-group for $B/I\rtimes_r G$ is $\Z^4$.
\end{lem}
\begin{proof}
Consider the Pimsner-Voiculescu six-term exact sequence for the action of $G$ on $B/I$:
\begin{align*}
\xymatrix{\K_0(B/I)^n  \ar[r]^-{\sigma}  & \K_0(B/I)  \ar[r] & \K_0(B/I\rtimes_r G) \ar[d]\\
\K_1(B/I\rtimes_r G) \ar[u]_-{{\delta}}  & \K_1(B/I) \ar[l]_-{{\bold 0}} & \K_1(B/I)^n. \ar[l]}
\end{align*}

Here, the map $\sigma$ denotes $\sum_i (1-t_i)$ with $t_1=a, t_2=b$ and $\delta$ is the index map. Recall that separable $\AF$-algebras (in particular $B/I$) have trivial $\K_1$. Using exactness we obtain that $\delta$ is injective ($\textrm{im} ({\bold 0})=\textrm{ker} (\delta)$) and hence that
\begin{align*}
	\K_1(B/I\rtimes_r G) & \cong \textrm{im} (\delta)= \textrm{ker} (\sigma)\\
	& = \{(h,h')\in \K_0(B/I)^2\colon \sigma(h,h')=0 \}.
\end{align*}	
Fix $(h,h')\in \K_0(B/I)^2$. Find projections $e,f,e',f'$ in $M_m(B/I)$ such that $h=[e]_0-[f]_0$ and $h'=[e']_0-[f']_0$. Suppose that $\sigma(h,h')=0$, i.e., $([e]_0-[f]_0)-a.([e]_0-[f]_0)+([e']_0-[f']_0)-b.([e']_0-[f']_0)=0$. Let us now use that $\K_0(B/I) \cong \Cy(\widehat{B/I}, \Z)$ by the isomorphism $\dim([e]_0)(x)=\textrm{Tr}(e(x))$, where Tr is the non-normalized trace map, to obtain that
$$p+q=a.p+b.q,$$
with $p=\textrm{Tr}(e)-\textrm{Tr}(f)$ and $q=\textrm{Tr}(e')-\textrm{Tr}(f')$.
Let $r,r'\in B/ I$ denote the two non-trivial projections from Lemma~\ref{lem16} fulfilling that $a.r=r$, $b.r'=r'$. Using Lemma~\ref{lem16} we have that $p=n1+ mr$ and $q=k1+ lr'$ for some $n,m,k,l\in \Z$. We conclude that $(h,h')=(n[1]_0+m[r]_0,k[1]_0+l[r']_0)$. We now have
\begin{align*}
	\K_1(B/I\rtimes_r G) & \cong \{(h,h')\in \K_0(B/I)^2\colon \sigma(h,h')=0 \}\\
	& =\{(n[1]_0+m[r]_0,k[1]_0+l[r']_0)\colon  n,m,k,l\in \Z\} \cong  \Z^4.
\end{align*}	
\end{proof}

\emph{Proofs of Theorem~\ref{mainth0}, Theorem~\ref{mainth1}, and Theorem~\ref{mainth2}:} By Proposition~\ref{prop21} the results stated in Theorem~\ref{mainth0}, Theorem~\ref{mainth1} and Theorem~\ref{mainth2} are an immediate consequence of Lemmas~\ref{lem-1}--\ref{lem8} and, Lemmas~\ref{lem11}--\ref{lem17} (cf.~\cite{Sier:Ror}). Both crossed products are the \Cs{}s of amenable groupoids, and such algebras have been proved to belong to the UCT class by Tu in \cite[Proposition 10.7]{Tu}.
\qed

\section{The $\K$-theory of the Roe algebra}
Let $G$ be a countable group and let $\elll^\infty(G)\rtimes_r G$ ($\cong \Cy(\beta G) \rtimes_r G$) denote the corresponding \emph{Roe algebra} crossed product\cp~\cite[p.\  152]{Hig:Roe}. This algebra encodes many crucial properties of the group. It is well known that the Roe algebra $\elll^\infty(G)\rtimes_r G$ is nuclear precisely when $G$ is exact \cite{Del1, Nar:Oza}. The crossed product is properly infinite if, and only if, $G$ is non-amenable\cp~\cite[Theorem 2.5.1]{Sier:phd}. Moreover, if a subset $E$ of $G$ is $G$-paradoxical, then $1_E \in \elll^\infty(G)$ is properly infinite $\elll^\infty(G)\rtimes_r G$. Remarkably, in \cite[Proposition 5.5]{Sier:Ror}, the converse was shown to be true. This observation leads to an entirely new way to tackle open problems regarding the Roe algebra. In particular we show that it is possible to determine part of the structure of $\K_0(\elll^\infty(G)\rtimes_r G)$ when $G$ is non-amenable. Our result is:
\begin{thrm}\label{mainth4}
Let $G$ be a non-amenable discrete group. Then every projection in $\elll^\infty(G)$ belongs to the trivial $K_0$-class of $K_0(\elll^\infty(G)\rtimes_r G)$. 

In particular, if $G$ is the free group on $1<n<\infty$ generators then $K_0(\elll^\infty(G)\rtimes_r G)=0$.
\end{thrm}

\begin{question}
Is $\K_0(\elll^\infty(G)\rtimes_r G)=0$ for every non-amenable group $G$?
\end{question}

Before we give a proof of Theorem~\ref{mainth4} let us establish some preliminary results.

\begin{lem}\label{lem01}
Let $G$ be a discrete group. Every projection in $\elll^\infty(G)$ is a sum of three projections having a complement that is full (i.e., not contained in a proper $G$-invariant closed two-sided ideal in $\elll^\infty(G) \rtimes_r G$).
\end{lem}
\begin{proof}
Find a partition $e, f, h\in \elll^\infty(G)$ of the unit $1$ and $t\in G$ such that each of these projections is orthogonal to its translate by $t$; see \cite[Corollary 6.2]{Sier:Ror}. Using the inequality $t.e\leq f+h=1 -e$ we obtain that the closed two-sided ideal in $\elll^\infty(G)\rtimes_r G$ generated by $1 -e$ contains $t.e$ (and $e$). Hence $1 -e$ is a full projection in $\elll^\infty(G)\rtimes_r G$.

Now fix a projection $p\in \elll^\infty(G)$. Partition $p$ into subprojections $p_e, p_f, p_h$ below $e, f, h$. Since $1-e$ is full in $\elll^\infty(G)\rtimes_r G$ and $1-e\leq 1-p_e$ we also have that $1-p_e$ is full. In particular, $p=p_e+p_f+p_h$ is a sum of three projections with full complements, $1-p_e,1-p_f,1-p_h$.
\end{proof}

\begin{lem}\label{lem03}
Let $G$ be a non-amenable discrete group. Then every projection in $\elll^\infty(G)$ that is full in ${\elll^\infty(G)\rtimes_r G}$ belongs to the trivial $K_0$-class.
\end{lem}
\begin{proof}
Assume first that $G$ is countable. Fix any projection $p\in \elll^\infty(G)$ that is full in ${\elll^\infty(G)\rtimes_r G}$. By \cite[Corollary 5.6]{Sier:Ror} the projection $p$ is properly infinite in ${\elll^\infty(G)\rtimes_r G}$. By Lemma~\ref{men31}, below, there exist projections $e,f\in \elll^\infty(G)$ that are equivalent to $p$ and add up to $p$. Hence $[p]_0=0$.
	
Now let $G$ be arbitrary (non-amenable). It suffices to extend \cite[Corollary 5.6]{Sier:Ror}---or, rather, \cite[Proposition 5.5]{Sier:Ror}---to this case. In order to do this, it turns out to be sufficient to replace the statement in \cite{Sier:Ror} of Dini's theorem for increasing sequences by the statement of it for arbitrary increasing nets (for which the proof is the same).
\end{proof}

\begin{lem}\label{lem04}
Let $G$ be a non-amenable discrete group. Then every projection in $\elll^\infty(G)$ belongs to the trivial $K_0$-class of $K_0(\elll^\infty(G)\rtimes_r G)$. 
\end{lem}
\begin{proof}
Let $p$ be a projection in $\elll^\infty(G)$. Using Lemma~\ref{lem01} find projections $p_e, p_f, p_h$ that add up to $p$ and have a full complement. Using Lemma~\ref{lem03} we have that
$$[1]_0 = 0,\ \ [1-p_e]_0 = 0,\ \ [1-p_f]_0 = 0,\ \ [1-p_h]_0 = 0.$$
We conclude that $[p]_0 = 0$.
\end{proof}

\emph{Proof of Theorem~\ref{mainth4}:}
Let $G$ denote the free group on $n$ generators. Consider the Pimsner-Voiculescu six-term exact sequence for the action of $G$ on $\elll^\infty(G)$:
\begin{align*}
\xymatrix{\K_0(\elll^\infty(G))^n  \ar[r]  & \K_0(\elll^\infty(G))  \ar[r]^-{{\bold 0}} & \K_0(\elll^\infty(G)\rtimes_r G) \ar[d]^-{{\bold 0}}\\
\K_1(\elll^\infty(G)\rtimes_r G) \ar[u]  & \K_1(\elll^\infty(G)) \ar[l] & \K_1(\elll^\infty(G))^n \ar[l]}
\end{align*}
Using that $\K_1(\cM)=0$ for any von Neumann algebra together with Lemma~\ref{lem04} we have two zero maps, as indicated above. Here we also use that $p\sim p_{11}\oplus \dots \oplus p_{nn}$ for a projection $p\in M_n(\elll^\infty(G))$\cp~proof of Lemma~\ref{lem6}. Using exactness we obtain that $$0= \K_1(\elll^\infty(G))^n = \textrm{im} ({\bold 0}) = \textrm{ker} ({\bold 0}) = \K_0(\elll^\infty(G)\rtimes_r G).$$
\qed

\section{Refining the construction of \cite{Sier:Ror}}\label{sec5}
It was shown in \cite[Theorem 6.11]{Sier:Ror} that every discrete countable non-amenable exact group admits a (free, amenable, minimal) action on the Cantor set such that the corresponding crossed product \Cs{} is a Kirchberg algebra in the $\UCT$ class. We have the following strengthened version of \cite[Theorem 6.11]{Sier:Ror}.

\begin{thrm}
\label{mainth3}
Let $G$ be a countable discrete group. Then $G$ admits a free, amenable, minimal action on the Cantor set $X$ such that $C(X)\rtimes_r G$ is a Kirchberg algebra in the UCT class if, and only if, $G$ is exact and non-amenable.

The action may be chosen in such a way that $[p]_0 = 0$ in $K_0(C(X)\rtimes_r G)$ for every projection $p$ in $C(X)$. In particular, if $K_0(C(X))\to K_0(C(X)\rtimes_r G)$ is surjective then $K_0(C(X)\rtimes_r G)=0$.
\end{thrm}

Before we give a proof of Theorem~\ref{mainth3} let us establish some preliminary results. The core idea is contained in the following lemma:

\begin{lem} \label{men31}
Let $G$ be a discrete group and let $\cP$ denote the set of all projections in $\elll^\infty(G)$. Then a projection $p\in \cP$ is properly infinite in $\elll^\infty(G) \rtimes_r G$ if, and only if, there exist $v,w\in C_c(G,\cP)$ such that $$p=vv^*=ww^*=v^*v+w^*w, \ \ \ v^*v,w^*w \in \elll^\infty(G).$$
\end{lem}
\begin{proof} 
Suppose that $p=1_E\in \cP$ is properly infinite in $\elll^\infty(G) \rtimes_r G$. By \cite[Proposition 5.5]{Sier:Ror} the set $E\subseteq G$ is $G$-paradoxical, i.e  there exist non-empty sets $V_1,V_2,\dots$, $V_{n+m}$ of $G$ and elements $t_1,t_2,\dots,t_{n+m}$ in $G$ such that $\bigcup_{i=1}^nV_i  = \bigcup_{i=n+1}^{n+m}V_i  = E$ and such that $\big(t_k. V_k\big)_{k=1}^{n+m}$ are pairwise disjoint subsets of $E$\cp~\cite[Definition 1.1]{Wag:B-T}. Using the Banach-Schr\"oder-Bernstein Theorem (cf.~\cite[Theorem 3.5]{Wag:B-T}) one may moreover assume that $\big(V_k\big)_{k=1}^{n}$ are pairwise disjoint, that $\big(V_{n+k}\big)_{k=1}^{m}$ are pairwise disjoint, and that $\bigcup_{i=1}^{n+m} t_i.V_i=E$. Following the proof of \cite[Proposition 4.3]{Sier:Ror}, construct the desired partial isometries $v,w\in \Cy_c(G,\cP)$ such that $p=vv^*=ww^*=v^*v+w^*w$ and $v^*v,w^*w \in \elll^\infty(G)$.
\end{proof}

Using Lemma~\ref{men31} we obtain a strengthened version of \cite[Proposition 6.8]{Sier:Ror} as follows:

\begin{prop} \label{Roe-2}
Let $G$ be a countable discrete group, and let $N$ be a countable subset of $\elll^\infty(G)$. Then there exists a separable $G$-invariant sub-\Cs{} $A$ of $\elll^\infty(G)$ which is generated by projections and contains $N$, with the following property: for every projection $p$ in $A$, if $p$ is properly infinite in $\elll^\infty(G) \rtimes_r G$, then $p$ is properly infinite in $A \rtimes_r G$.

One can refine the construction is such a way that each properly infinite projection $p$ is a sum of two projections in $A$ equivalent to $p$ in $A\rtimes_r G$.
\end{prop}
\begin{proof} 
Let $\cP_\mathrm{inf}$ denote the set of properly infinite projections in $\elll^\infty(G) \rtimes_r G$. Use \cite[Lemma 6.7]{Sier:Ror} to find a countable $G$-invariant set of projections $P_0 \subseteq \elll^\infty(G)$ such that $N \subseteq \cstar(P_0)$. Let $Q_0$ denote the set of projections in $\cstar(P_0)$. The set $Q_0$ is countable because $\cstar(P_0)$ is separable and abelian\cp~proof of Lemma~\ref{lem1}. For each $p \in Q_0 \cap\cP_\mathrm{inf}$ use Lemma~\ref{men31} to find a countable subset $M(p)$ of $\elll^\infty(G)$ such that $p$ is properly infinite in $A\rtimes_r G$, and is in fact a sum of two projections in $A$ equivalent to $p$ in $A\rtimes_r G$, whenever $A$ is a $G$-invariant sub-\Cs{} of $\elll^\infty(G)$ that contains $\{p\} \cup M(p)$. Put 
$$N_1 = Q_0 \, \cup \, \bigcup_{p \in Q_0 \cap\cP_\mathrm{inf}} M(p).$$ 
Use Lemma \cite[Lemma 6.7]{Sier:Ror} to find a countable $G$-invariant set of projections $P_1 \subseteq \elll^\infty(G)$ such that $N_1 \subseteq \cstar(P_1)$. 

Continue in this way to find countable subsets $N_0=N, N_1,N_2, \dots$ of $\elll^\infty(G)$ and countable $G$-invariant subsets $P_0,P_1,P_2,\dots$ consisting of projections in $\elll^\infty(G)$ such that if $Q_j$ is the (countable) set of projections in $\cstar(P_j)$, then $Q_{j} \subseteq N_{j+1}\subseteq \cstar(P_{j+1})$ and every $p \in Q_{j} \cap \cP_\mathrm{inf}$ is properly infinite in $\cstar(P_{j+1})\rtimes_r G$, and is in fact a sum of two projections in $\cstar(P_{j+1})$ equivalent to $p$ in $\cstar(P_{j+1})\rtimes_r G$.

Put $P = \bigcup_{n=0}^\infty P_n$ and put $A = \cstar(P)$. Notice that
$$A = \overline{\bigcup_{n=1}^\infty \cstar(P_n)}.$$

Then $P$ is a countable $G$-invariant subset of $\elll^\infty(G)$ consisting of projections, $N \subseteq \cstar(P)$. Moreover, if $p$ is a projection in $A$ which is properly infinite in $\elll^\infty(G)\rtimes_r G$, then $p$ is equivalent (and hence equal to) a projection in $\cstar(P_n)$ for some $n$, whence $p$ belongs to $Q_n \cap\cP_{\mathrm{inf}}$, which by construction implies that $p$ is properly infinite in $\cstar(P_{n+1}) \rtimes_r G$ and hence also in $A \rtimes_r G$. 

Moreover, the projection $p$ is a sum of two projections $\cstar(P_{n+1})$ equivalent to $p$ in $\cstar(P_{n+1})\rtimes_r G$. Hence, $p$ is a sum of two projections in $A$ equivalent to $p$ in $A\rtimes_r G$.
\end{proof}

\emph{Proof of Theorem~\ref{mainth3}:} Following the proof of \cite[Theorem 6.11]{Sier:Ror} use Proposition~\ref{Roe-2} in stead of \cite[Proposition 6.8]{Sier:Ror} (together with the fact that homomorphisms preserve equivalence). This ensures that every projection $p$ in $\Cy(X)$ is a sum of two projections in $\Cy(X)$ equivalent to $p$ in $\Cy(X)\rtimes_r G$. In particular, $[p]_0=0$ in $\K_0(\Cy(X)\rtimes_r G)$.

Suppose that $\iota\colon \K_0(\Cy(X))\to \K_0(\Cy(X)\rtimes_r G)$ is surjective. Following the proof of Lemma~\ref{lem15} we conclude that $\K_0(\Cy(X)\rtimes_r G)=0$.
\qed

\begin{remark}
There is another way to strengthen \cite[Theorem 6.11]{Sier:Ror}. One can carry out the construction in such a way that $\K_0(\Cy(X)\rtimes_r G)=0$ provided that $G$ has the following two properties: First, $\K_0(\elll^\infty(G) \rtimes_r G)=0$ and, second, whenever $G$ acts on an abelian \Cs{} $A$ of real rank zero then the map $\K_0(A\rtimes_r G)\to \K_0(A/I\rtimes_r G)$ is surjective for every $G$-invariant closed two-sided ideal $I$ in $A$.

Recall that the crossed product $\Cy(X)\rtimes_r G$ in \cite[Theorem 6.11]{Sier:Ror} is constructed as the limit of a sequence of separable properly infinite algebras $A_i\rtimes_r G$, divided by an ideal $I\rtimes_r G$. The algebra $A=\underrightarrow{\lim} A_i$ is a unital abelian separable \Cs{} of real rank zero. Hence each $\K_0(A_i\rtimes_r G)$ arises from a countable set of projections in $Z_i\subseteq A_i\rtimes_r G$\cp~\cite{RorLarLau:k-theory}. We conclude that $\K_0(A\rtimes_r G)$ arises from the set of projections $\bigcup_{i=1}^\infty Z_i$\cp~\cite{RorLarLau:k-theory}. We therefore need to ensure that each $p\in Z_i$ has trivial $\K_0$-class.

Let $p$ be a projection in $\elll^\infty(G)\rtimes_r G$ such that $[p]_0=0$ in $\K_0(\elll^\infty(G)\rtimes_r G)$. Because $\Cy_c(G,\elll^\infty(G))$ is dense in $\elll^\infty(G)\rtimes_r G$ there exists a countable subset $N(p)$ of $\elll^\infty(G)$ such that whenever $A$ is a $G$-invariant sub-\Cs\ of $\elll^\infty(G)$ which contains $N(p)$, then $[p]_0=0$ in $\K_0(A\rtimes_r G)$\cp~\cite[Lemma 6.6]{Sier:Ror}. Again by a Blackadar-type argument, one can add these sets $N(p)$ to the inductive construction of the $A_i$'s in such a way that we obtain a new version of \cite[Proposition 6.8]{Sier:Ror} with the following additional statement: If $\K_0(\elll^\infty(G)\rtimes_r G)=0$ then $\K_0(A\rtimes_r G)=0$. Since $\K_0(\elll^\infty(G)\rtimes_r G)=0$ by assumption, we have that $\K_0(A\rtimes_r G)=0$, and hence again by hypothesis, $\K_0(A/I\rtimes_r G)=0$, i.e., $\K_0(\Cy(X)\rtimes_r G)=0$, as desired.
\begin{question}
Can we ensure $A\rtimes_r G$ is $K_0$-liftable\cp~\cite{PasRor}?
\end{question}
\end{remark}

\begin{cor}\label{cor1}
Let $G$ denote the free group on $n$ generators. Then $G$ admits a free, amenable, minimal action on the Cantor set $X$ such that $C(X)\rtimes_r G$ is a Kirchberg algebra in the UCT class and $K_0(C(X)\rtimes_r G)=0$.
\end{cor}

\proof[Acknowledgements]
We thank the members or the operator algebra seminar group at the Fields Institute, in particular Barry Rowe, for valuable discussions concerning the topics of this paper.

\providecommand{\bysame}{\leavevmode\hbox to3em{\hrulefill}\thinspace}
\providecommand{\MR}{\relax\ifhmode\unskip\space\fi MR }

\end{document}